\documentclass[11pt]{article}

\usepackage{amsmath,amssymb,amsthm,amsfonts,latexsym,tikz,hyperref,shuffle}
\usepackage[hmargin=1in,vmargin=1in]{geometry}
\usepackage[notref, notcite]{showkeys}
\usepackage{ytableau}
\usepackage{comment}
\usepackage[normalem]{ulem}
\usepackage{tikz}
\usepackage{mathtools}

\DeclarePairedDelimiter{\ceil}{\lceil}{\rceil}

\DeclareMathOperator{\cPin}{cPin}
\DeclareMathOperator{\Pin}{Pin}

\newcommand{\floor}[1]{\lfloor #1 \rfloor}

\makeatletter
\providecommand*{\shuffle}{%
  \mathbin{\mathpalette\shuffle@{}}%
}
\newcommand*{\shuffle@}[2]{%
  \sbox0{$#1\vcenter{}$}%
  \kern .15\ht0 
  \rlap{\vrule height .25\ht0 depth 0pt width 2.5\ht0}%
  \raise.1\ht0\hbox to 2.5\ht0{%
    \vrule height 1.75\ht0 depth -.1\ht0 width .17\ht0 %
    \hfill
    \vrule height 1.75\ht0 depth -.1\ht0 width .17\ht0 %
    \hfill
    \vrule height 1.75\ht0 depth -.1\ht0 width .17\ht0 %
  }%
  \kern .15\ht0 
}
\makeatother

\newtheorem{thm}{Theorem}[section]
\newtheorem{prop}[thm]{Proposition}
\newtheorem{cor}[thm]{Corollary}
\newtheorem{lem}[thm]{Lemma}
\newtheorem{conj}[thm]{Conjecture}
\newtheorem{exa}[thm]{Example}

\newtheorem{theorem}[thm]{Theorem}

\newtheorem{maybe theorem}[thm]{Maybe Theorem}

\newcommand{\ben}{\begin{enumerate}}
\newcommand{\een}{\end{enumerate}}
\newcommand{\ble}{\begin{lem}}
\newcommand{\ele}{\end{lem}}
\newcommand{\bth}{\begin{thm}}
\renewcommand{\eth}{\end{thm}}
\newcommand{\bpr}{\begin{prop}}
\newcommand{\epr}{\end{prop}}
\newcommand{\bco}{\begin{cor}}
\newcommand{\eco}{\end{cor}}
\newcommand{\bcon}{\begin{conj}}
\newcommand{\econ}{\end{conj}}
\newcommand{\bde}{\begin{defn}}
\newcommand{\ede}{\end{defn}}
\newcommand{\bex}{\begin{exa}}
\newcommand{\eex}{\end{exa}}
\newcommand{\barr}{\begin{array}}
\newcommand{\earr}{\end{array}}
\newcommand{\btab}{\begin{tabular}}
\newcommand{\etab}{\end{tabular}}
\newcommand{\beq}{\begin{equation}}
\newcommand{\eeq}{\end{equation}}
\newcommand{\bea}{\begin{eqnarray*}}
\newcommand{\eea}{\end{eqnarray*}}
\newcommand{\bal}{\begin{align*}}
\newcommand{\bce}{\begin{center}}
\newcommand{\ece}{\end{center}}
\newcommand{\bpi}{\begin{picture}}
\newcommand{\epi}{\end{picture}}
\newcommand{\bpp}{\begin{picture}}
\newcommand{\epp}{\end{picture}}
\newcommand{\bfi}{\begin{figure} \begin{center}}
\newcommand{\efi}{\end{center} \end{figure}}
\newcommand{\bprf}{\begin{proof}}
\newcommand{\eprf}{\end{proof}\medskip}

\newcommand{\bsl}{\begin{slide}{}}
\newcommand{\esl}{\end{slide}}
\newcommand{\bfr}{\begin{frame}}
\newcommand{\efr}{\end{frame}}

\newcommand{\hso}[1]{\hspace{-1pt}}

\def\<{\langle}
\def\>{\rangle}

\newcommand{\0}{{\bf 0}}
\newcommand{\1}{{\bf 1}}
\newcommand{\2}{{\bf 2}}
\newcommand{\3}{{\bf 3}}

\DeclareMathOperator{\ord}{ord}

\begin{document}

\title{Further Results on Pinnacle Sets}
\author{
Quinn Minnich\\[-1pt]
\small Department of Mathematics, Michigan State University,\\[-1pt]
\small East Lansing, MI 48824-1027, USA, {\tt minnichq@msu.edu}\\
}

\date{\today\\[10pt]
	\begin{flushleft}
	                                       \\[5pt]
	\end{flushleft}}

\maketitle

\begin{abstract}

The study of pinnacle sets has been a recent area of interest in combinatorics. Given a permutation, its pinnacle set is the set of all values larger than the values on either side of it. Largely inspired by conjectures posed by Davis, Nelson, Petersen, and Tenner and also results proven recently by Fang, this paper aims to add to our understanding of pinnacle sets. In particular, we give a simpler and more combinatorial proof of a weighted sum formula previously proven by Fang. Additionally, we give a recursion for counting the admissible orderings of the elements of a potential pinnacle set. Finally, we give another way of viewing pinnacle sets that sheds light on their structure and also yields a recursion for counting the number of permutations with that pinnacle set.

\end{abstract}

\section{Introduction}

Given a permutation $\pi$ on some subset of the elements in $[n]$, we define its \textit{pinnacle set} to be
\begin{equation}
    \label{Pin}
   \Pin \pi = \{\pi_i \mid \pi_{i-1} < \pi_{i} > \pi_{i+1}\}. 
\end{equation}
In other words, the pinnacle set of $\pi$ is the set of all elements of $\pi$ that are between two values smaller than itself. This is similar to the peak set of $\pi$, except that the peak set keeps track of the indices of such values rather than the values themselves. Let $S_n$ be the symmetric group on $n$ elements. Then given a set $P \subseteq [n]$, we say that $P$ is an \textit{admissible pinnacle set} or \textit{pinnacle set} for short if there exists some $\pi \in S_n$ such that $\Pin \pi = P$. 

Some of the initial work related pinnacle sets (although under a different name) was done by Strehl in \cite{Strehl_1978} where he studied permutations in which pinnacles and non-pinnacles alternated. The concept was rediscovered independently by Davis, Nelson, Petersen, and Tenner in \cite{DNPT_2018} where they proved several initial results, including that the number of admissible pinnacle sets $P \subseteq [n]$ is counted by a central binomial coefficient. Not long afterwards, Domagalski et al.\ gave a formula in \cite{DLMSSS_2021} for the number of permutations in $S_n$ having a given pinnacle set $P$ using the principle of inclusion and exclusion. Around the same time, Diaz-Lopez et al.\ gave a different formula for the same number by keeping track of \textit{vale sets}, which are those elements of a permutation $\pi$ smaller than those on either side. Both of these formulas however ran in time exponential in the number of elements of $P$. 

Shortly afterwards, Falque et al.\
published some results in \cite{FNT_2021} which included a recursive formula for the number of permutations having a fixed pinnacle set that ran in time $nk^2$ where $k$ was the number of elements in $P$ and $n$ was the length of the permutations being counted. Finally, Fang proved a result in \cite{Fang_2021} which computed the number in $O(k^2\log n + k^4)$ time.

In addition to results that pertained to the number of permutations with a given pinnacle set, Rusu in \cite{Rusu_2020} studied transforming a permutation with a given pinnacle set into another permutation with the same pinnacle set such that the pinnacle set was preserved in all intermediate steps if the transformation. Additional work was also done by Rusu and Tenner in \cite{RT_2021} characterizing the orderings in which the elements of some pinnacle set $P$ could appear within a permutation with pinnacle set $P$. 

A \textit{cyclic permutation} for some $\pi = \pi_1\pi_2\cdots\pi_n$, denoted $[\pi]$, is the set of all rotations
$$[\pi] = \{\pi_1\pi_2\cdots\pi_n, \pi_n\pi_1\cdots\pi_{n-1}, \ldots, \pi_2\cdots\pi_n\pi_1\}.$$
Alternately, one may think of the cyclic permutation $[\pi]$ as being the linear permutation $\pi$ where the last elements is considered to be adjacent to the first. The pinnacle set of a cyclic permutation $[\pi]$, denoted $\cPin [\pi]$, is then defined to be as in equation~\eqref{Pin} with the subscripts taken modulo $n$. Let $[S_n]$ be the set of all cyclic permutations on $n$ elements. The pinnacle sets of linear and cyclic permutations are related in the following way. Given some linear permutation $\pi \in S_n$ with pinnacle set $P$, we may add the element $n+1$ to the end of $\pi$ and wrap it around into a circle to get $[\pi']$, which must then have pinnacle set $P' = P \cup \{n+1\}$. Therefore, when deriving results about pinnacle sets of linear permutations, we may often switch to the cyclic version of the problem, which can eliminate boundary cases.

Cyclic permutations have been the focus of a lot of recent study. Pattern avoidance in the case of cyclic permutations was examined by Vella in \cite{Vella_2002} and also independently by Callan in \cite{Callan_2002}. Further results involving patterns in cyclic permutations were given by Gray, Lanning, and Wang in \cite{GLW_2018, GLW_2019}; by Czabarka and Wang in \cite{CW_2019} who proved a cyclic version of the Erd\H{o}s-Szekeres Theorem; and also by Domagalski et al.\ in \cite{DLMSSS_jun_2021}. Elizalde and Sagan then expanded on the results of Domagalski et al.\ in \cite{EB_2021} where they proved some of the remaining conjectures. In \cite{AGRR_2020}, Adin et al.\ studied cyclic quasi-symmetric functions which also related to cyclic shuffling of permutations. Finally, Domagalski et al.\ used circular permutations in \cite{DLMSSS_2021} to prove some results about pinnacle sets, including the formula mentioned above. 

It was first proven by Davis, Nelson, Petersen, and Tenner in \cite{DNPT_2018} that $P$ is a pinnacle set of a linear permeation if and only if, for every $p \in P$, the number of pinnacle elements in $P$ less than or equal to $p$ is exceeded by the number of elements not in $P$ that are less than $p$. But while it is easy to determine if a given set $P$ is a pinnacle set, many other aspects about pinnacle sets seem difficult to analyze. One question that has received a good deal of interest has been counting the number of permutation $\pi \in S_n$ that have a given pinnacle set $P \subseteq S_n$. To this end, we define 

$$p_n (P) = \{\pi \in S_n \mid \Pin (\pi) = P\}$$ 

\noindent to be the set of permutations with a pinnacle set $P$. Then if for some set $S$ we let $|S|$ denote the carnality of $S$, we have that $|p_n(P)|$ is the number of permutations with pinnacle set $P$. In the event that $P$ is not a pinnacle set, we will simply have $p_n(P)=\emptyset$ and throughout this paper, we assume that $P$ is an arbitrary subset of $[n]$ and do not require it to necessarily be an admissible pinnacle set. Multiple formulas and recursions for $|p_n (P)|$ have been proven, including some in \cite{DNPT_2018, DLMSSS_2021, DHHIN_2021, FNT_2021, Fang_2021}.

In \cite{Fang_2021}, Fang proved a formula for the weighted sum of the $|p_n(Q)|$ over all $Q \subseteq P$, but the proof was highly computational and involved Lagrange interpolation. Additionally, his formula needed to assume that $P$ was an admissible pinnacle set. In section 1 of this paper, we present a simpler, more combinatorial proof of his formula that also works for any $P \subseteq [n]$ by forming a bijection between elements of the sum and a weighted version of Motzkin walks.

Given an arbitrary set $P \subset [n]$, we can also study the ordering of the elements of $P$ within any $\pi$ having $\Pin(\pi) = P$. In some cases, only some orderings of $P$ will appear in any of the $\pi \in p_n(P)$, and we call all such orderings the \textit{admissible orderings of $P$}. In \cite{DNPT_2018} the question was raised as to whether there was a way to count the number of admissible orderings of a given $P$, but to our knowledge, no general formulas have been given. In section 2, we present a recursion that can find the number of admissible orderings of any set $P \subseteq [n]$ with a run time that depends only on $|P|$ and not $n$. 

Finally, in section 4 we give an alternate representation of pinnacle sets in the form of blocks consisting of $1$'s and $0$'s inspired by a suggestion in \cite{FNT_2021} which sheds some light on the internal structure of pinnacle sets. We then use this representation to prove a recursion that can find $|p_n(P)|$ in run time equal to the algorithm presented in \cite{Fang_2021}, which is currently the fastest known algorithm, and that can be further used to prove a version of a formula conjectured in \cite{FNT_2021}. 

\section{A formula for the weighted sum of $|p_n(P)|$}

Define a \textit{Dyck path} with $k$ steps to be a path starting at $(0,0)$, ending at $(k,0)$, that uses only down steps $[1,-1]$ and up steps $[1,1]$, and that never drops below the $x$-axis. A \textit{Dyck walk} of length $k$ is a generalised version of a Dyck path where we allow the path to end at any point $(k,y)$ for $y\geq 0$. We then define a \textit{modified Dyck walk} to be a Dyck walk where the last step may possibly drop below the $x$-axis. In other words, a modified Dyck walk is a walk starting at $(0,0)$ consisting of only up steps and down steps that may end at any height $h \geq -1$, and that does not drop below the $x$-axis except possibly at the final step. Define $R_k$ be the set of $y$-coordinate sequences of modified Dyck walks of length $k$. A walk $r\in R_k$ takes the form $(r_0, r_1, \ldots r_k)$ where $r_i$ is the height of the end of the $i$th step and $r_0 = 0$. From this, we have that $r_i \geq 0$ for all $0\leq i < k, \, r_k\geq -1,$ and $r_{i} = r_{i-1} \pm 1$ for all $1 \leq i \leq k$. 

We will give a generalization of Theorem 1.2 found in \cite{Fang_2021} with a simpler, combinatorial proof.

\begin{theorem}
\label{Main result 1}
For $n\geq 1$ and $P = \{p_1>p_2> \cdots >p_k\} \subseteq [n]$, we take the convention that $p_0 = n+1$ and $p_{k+1} = 1$. Then we have the following.
$$\sum_{Q \subseteq P} 2^{|Q|+1}|p_n(Q)| = 2^{n-k} \sum_{r\in R_k} \prod_{i=0}^k(r_i+1)^{p_i-p_{i+1}}.$$
\end{theorem}

This will be proven later, but first we need to set up some combinatorial objects.

\begin{figure}
    \begin{center}
    \begin{tikzpicture}
    \draw[help lines] (0,0) grid (9,2);,
        \draw (0,0)--(1,1)--(2,1)--(3,0)--(4,0)--(5,1)--(6,1)--(7,0)--(8,0)--(9,0);
        \node [above] at (0.5,0.5) {$1_r$};
        \node [above] at (1.5,1) {$1_l$};
        \node [above] at (2.5,0.5) {$2_l$};
        \node [above] at (3.5,0) {$1_r$};
        \node [above] at (4.5,0.5) {$1_r$};
        \node [above] at (5.5,1) {$2_l$};
        \node [above] at (6.5,0.5) {$2_l$};
        \node [above] at (7.5,0) {$1_l$};
        \node [above] at (8.5,0) {$1_r$};
        
        \node [below] at (0.5,0) {$9$};
        \node [below] at (1.5,0) {$8$};
        \node [below] at (2.5,0) {$7$};
        \node [below] at (3.5,0) {$6$};
        \node [below] at (4.5,0) {$5$};
        \node [below] at (5.5,0) {$4$};
        \node [below] at (6.5,0) {$3$};
        \node [below] at (7.5,0) {$2$};
        \node [below] at (8.5,0) {$1$};
    \end{tikzpicture}
    \caption{An example of a walk in $M_9(P)$ for $p = \{9,7,5,3\}$}
    \label{Example 1}
    \end{center}
\end{figure}
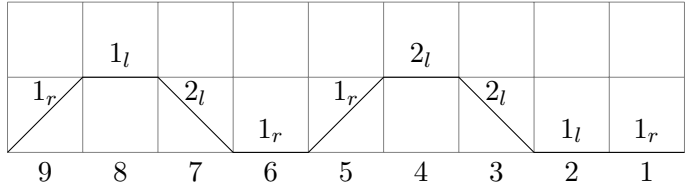

Define the set of \textit{decorated Motzkin walks}, $M_n (P)$, to be the set of all lattice walks with $n$ steps, starting at $(0,0)$, consisting of steps $[1,-1],[1,0],$ and $[1,1]$, such that every step starts weakly above the $x$-axis, and where the steps are numbered $1,2, \ldots, n$ from right to left. Note that this will allow walks that can go below the $x$-axis only on their final step. The generating function for those paths of $M_n(P)$ whose last step does not go below the $x$-axis, along with the generating function for much more general Motzkin walks, has been studied by Alexios in \cite{Poly_2021} although for our purposes here we will not need the generating function.

Furthermore, we require that in each walk of $M_n(P)$ the step numbered $k$ be either the up step $[1,1]$ or the down step $[1,-1]$ if $k\in P$, and the horizontal step $[1,0]$ if $k \not \in P$. Finally, we give every step of the walk two labels. The first is an integer in $[h+1]$ where $h$ is the height of the starting point of the step. The second label is either a right or left designation subject to the condition that down steps are always labeled as left and up steps always labeled as right. An example of an element in $M_n(P)$ is given in Figure \ref{Example 1} where the number of each step is at the bottom of the figure and the label of each step is directly above the step. Note that for simplicity, we have combined both labels by giving $l$ and $r$ subscripts to the height labels. For instance, the label given to the third step in Figure \ref{Example 1} is $2_l$ meaning it has a height label of $2$ and a left designation. 

The count for the number of walks in $M_n(P)$ is given by the following result.

\begin{theorem}
Suppose $n\geq 1$ and $P \subseteq [n]$. Then the formula
$$|M_n(P)| = 2^{n-k} \sum_{r\in R_k} \prod_{i=0}^k(r_i+1)^{p_i-p_{i+1}}$$
gives the count for the number of distinct decorated Motzkin walks.
\end{theorem}
\begin{proof}
If we consider any walk in $M_n(P)$ we note that by removing all horizontal steps we reduce the walk to one in $R_k$. We may also reverse this process by starting with a walk $r\in R_k$ and adding horizontal steps as specified by $P$. Therefore, to find $|M_n(P)|$ we need only sum over all sequences $r \in R_k$ and for each resulting walk, determine how many different labelings are possible. 

Consider all walks in $M_n(P)$ corresponding to a fixed $r\in R_k$. Note that the height of the step corresponding to $p_{i+1}$ and the heights of the steps between $p_i$ and $p_{i+1}$ will have starting point at height $r_i$ since the height can only change on a step corresponding to an element in $P$. Therefore there are $r_i+1$ choices for each of the labels of these $p_i-p_{i+1}$ steps. Additionally, since only horizontal steps have a choice for subscripts, we must multiply by a factor of 2 for each element not in $P$. Therefore, the total count is:
$$|M_n(P)| = 2^{n-k} \sum_{r\in R_k} \prod_{i=0}^k(r_i+1)^{p_i-p_{i+1}}$$
This completes the proof.
\end{proof}

Given a $\pi \in S_n$ and arbitrary $P\subseteq [n]$, let $[\pi']$ be the cyclic permutation obtained by adding the element $n+1$ to the end of $\pi$ and wrapping it around into a circle, and let $P' = P \cup \{n+1\}$. We then have that the cyclic permutations with $\cPin [\pi'] = P'$ are in bijection to those in $p_n(P)$, since this operation is invertible, and preserves all pinnacles except $n+1$. Define an \textit{intermediate set} to be the set of all elements between (but not including) two consecutive elements of $P'$ in $[\pi']$. Intermediate sets may be empty. Given $\pi$  a \emph{cyclic vale} is any vale of $[\pi']$. Because $n+1$ will always be a pinnacle, all cyclic vales will be in $[n]$ and will alternate with the pinnacles, which means there will be exactly $k+1$ of them in $[\pi']$ where $k$ is the number of pinnacles of $\pi$.

In order to prove Theorem \ref{Main result 1}, we will interpret the sum 
$$\sum_{Q \subseteq P} 2^{|Q|+1}|p_n(Q)|$$ 
to be counting the number of cyclic permutations such that $[\pi'] \in [S_{n+1}]$, $\cPin [\pi']  = Q' = Q \cup \{n+1\} \subseteq P'$, and where every element is marked as either left or right subject to certain restrictions. If we note that every element of a cyclic permutation is either a pinnacle, a vale, or in the middle of an increasing or decreasing sequence of three elements, we may impose the following labeling.
\begin{enumerate}
    \item Any element that is a cyclic vale may be marked as either right or left.
    \item Any element of $P$ that is not a cyclic vale is always marked as right.
    \item Any element not in $P$ that is in the middle of a decreasing sequence of 3 elements is always marked as right.
    \item Any element not in $P$ that is in the middle of an increasing sequence of 3 elements is always marked as left.
\end{enumerate} 
Let us call the set of all such permutations $V_n(P)$. Note that by this convention, all markings are forced except those at the cyclic vales, which are all free, and therefore the carnality of $V_n(P)$ is 
$$\sum_{Q \subseteq P} 2^{|Q|+1}|p_n(Q)|$$
since the number of cyclic vales is $|Q|+1$ and since the number of cyclic permutations with $\cPin [\pi'] = Q'$ is $|p_n(Q)|$.

Next, we define a map $f$ from $M_n(P)$ to $V_n(P)$. To do this, we read the steps of some $m\in M_n(P)$ one at a time and use it to build up a cyclic permutation. The process is described as follows.

\begin{enumerate}
    \item Start with the cyclic permutation consisting of only $p_0 = n+1$. In all future steps, an intermediate set will be said to be \textit{available} if neither of the elements in $P'$ that bound it have a left label. 
    To this end, $p_0$ is assumed to have a right label.
    \item Take the current leftmost step of walk $m$ and insert its step number, $i$, into the cyclic permutation in the following manner. Let the label of the step be $h_x$. Counting clockwise from $p_0$, insert $i$ into the $h$th available intermediate set and give it the left/right label corresponding to the subscript $x$. 
    If the intermediate set was empty, there is only one way to insert $i_x$.
    If the intermediate set is non-empty, look at the current smallest element of that set, $v$, and place $i$ adjacent to $v$ on the side corresponding to $v$'s label. Note that if $k$ was in $P'$, the available intermediate sets for future steps will be renumbered. 
    Finally, delete the step numbered $i$ from the walk.
    \item Repeat until all elements have been placed and the walk is empty.
\end{enumerate}

As an example of this algorithm, suppose that $P = \{9,7,5,3\}$ and consider the walk in Figure \ref{Example 1}. To compute the output of the map in $V_9(P)$ we successively generate the following permutations as we read the steps of $P$ from left to right, inserting the next element as indicated by the number and label of the step in the walk. Elements in $P$ are in bold to make it easier to see the intermediate sets.  Commentary is provided for the first few insertions to aid the reader.

\begin{align*}
    &[\mathbf{10_r}]\\[5pt]
    &[\mathbf{10_r},\mathbf{9_r}]: \text{$9_l$ is inserted in the only intermediate set available, splitting it in two.}\\[5pt]
    &[\mathbf{10_r},8_l,\mathbf{9_r}]: \text{$8_l$ is inserted in the first intermediate set out of the two since its step is labeled $1_l$.}\\[5pt]
    &[\mathbf{10_r}, 8_l,\mathbf{9_r},\mathbf{7_l}]: \text{$7_l$ is inserted in the second intermediate set because its label is $2_l$; from here to the end of}\\
    &\text{the algorithm, no other elements will be inserted between 9 and 7, or between 7 and 10 because $7 \in P$ and}\\
    &\text{step 7 has a left label.}\\[5pt]
    &[\mathbf{10_r},6_r,8_l,\mathbf{9_r},\mathbf{7_l}]: \text{$6_r$ is inserted to the left of the smallest element, $8$, of the first intermediate set because $8$}\\
    &\text{has a left label.}\\[5pt]
    &[\mathbf{10_r},6_r,\mathbf{5_r},8_l,\mathbf{9_r},\mathbf{7_l}]: \text{Inserting $5_r$ splits the only available intermediate set into two (one from $10$ to $5$ and}\\
    &\text{one from $5$ to $9$) since $5\in P$ and both are available since $5$ has a right label.}\\[5pt]
    &[\mathbf{10_r},6_r,\mathbf{5_r},4_l,8_l,\mathbf{9_r},\mathbf{7_l}]\\[5pt]
    &[\mathbf{10_r},6_r,\mathbf{5_r},\mathbf{3_l},4_l,8_l,\mathbf{9_r},\mathbf{7_l}]\\[5pt]
    &[\mathbf{10_r},6_r,2_l,\mathbf{5_r},\mathbf{3_l},4_l,8_l,\mathbf{9_r},\mathbf{7_l}]\\[5pt]
    &[\mathbf{10_r},6_r,1_r,2_l,\mathbf{5_r},\mathbf{3_l},4_l,8_l,\mathbf{9_r},\mathbf{7_l}]
\end{align*}

The last line is the final output.  

We will now prove some properties about this map so that we may show it is well-defined.

\begin{lem}
\label{labeling 1 for f} 
In the final result, all elements in $P'$ that are not cyclic vales are guaranteed to have right labels.
\end{lem}
\begin{proof}
This follows because when an element of $P'$ with a left label is initially inserted, it will become a cyclic vale.
Indeed, it is a cyclic vale upon insertion
since it is the smallest element inserted so far. And in all future steps no elements are ever inserted next to it on account of the intermediate sets on either side being unavailable.  So it will remain a cyclic vale until the end.
\end{proof}

\begin{lem}
\label{pinnacles for f}
The output of $f$ will always have pinnacles a subset of $P'$
\end{lem}
\begin{proof}
Consider some element $i \not \in P'$. Since all elements become vales when initially placed, it is enough to show that $i$ remains next to to one of its larger adjacent elements through the end of the algorithm. But this is immediate because if $i$ is labeled as left, then no future element will ever be inserted directly to the right of $i$, and similarly for a right label. Since no element is ever removed, we have that $i$ will remain next to a larger element and therefore fail to be a pinnacle.
\end{proof}

\begin{lem}
\label{labeling 2 for f}
In the final result, all elements $i \not \in P'$ that are also not vales will be in the middle of a decreasing sequence of three elements if $i$ has a right label, and in the middle of an increasing sequence of three elements if $i$ has a left label.
\end{lem}
\begin{proof}
First, since $i$ cannot be a vale by assumption and since it cannot be a pinnacle by Lemma \ref{pinnacles for f}, the only remaining possibilities are that it is in the middle of a decreasing or increasing sequence. If such an element $i$ has a left label, then the algorithm forces all future elements placed adjacent to $i$ to be to its left. Therefore, $i$ will always remain adjacent to the value initially on its right, which is larger than $i$. This makes it impossible for $i$ to end up in the middle of a decreasing sequence, and so it must be in the middle of an increasing one. A similar proof holds for when $i$ has a right label.
\end{proof}

\begin{lem}
\label{Lemma for f}
At any step of the algorithm, if $p_i$ was the last element of $P'$ to be placed, then there will be $r_i+1$ available intermediate sets for the next step.
\end{lem}
\begin{proof}
Suppose that element $i$ was just placed in $[\pi']$. If $i \in P'$ and corresponded to an up step, then $i$ will have a right label and the net effect on the permutation will be to replace the intermediate set it landed in with two new ones on either side, resulting in an increase of one available intermediate set. If $i\in P'$ corresponded to a down step however, it will have a left label and therefore make the intermediate set it lands in unavailable, for a net loss of one intermediate set. Finally, if $i \not \in P'$, then $i$ cannot change the number of intermediate sets since they are completely determined by the positions of the elements in $P'$. This, combined with the fact that at the start of the process there is always one available intermediate set and that $r_0 = 0$, is enough to show the result by induction.
\end{proof}

\begin{prop}
The algorithm defined by $f$ is well defined. In particular it always terminates,  and outputs a permutation in $V_n(P)$.
\end{prop}
\begin{proof}
To prove that $f$ terminates, note that there will always be at least one available intermediate set in which to place elements because every $r_i \geq 0$, for $0 \leq i <k$ and so by Lemma \ref{Lemma for f} there will always be at least one intermediate set at each step. Lemma \ref{pinnacles for f} shows that the pinnacles will be a subset of $P'$, and Lemmas \ref{labeling 1 for f} and \ref{labeling 2 for f} together show that those elements which are not cyclic vales will be marked as specified by the set $V_n(P)$. Finally, since the algorithm marks all cyclic vales, we have that $f$ is well-defined. 
\end{proof}

We now construct the inverse $g = f^{-1}$ by reversing each step of the algorithm for $f$. 

\begin{enumerate}
    \item Start with a permutation $[\pi'] \in V_n(P)$. We will build a walk $m\in M_n(P)$ from right to left with steps numbered $1,2,\ldots,n$. To build step number $i$, we remove the element $i$ from $[\pi']$ in the following way.
    
    \item Locate the smallest element, $i$, of $[\pi']$, and remove it to form $[\sigma]$. Now in $[\sigma]$ take note of the intermediate set, $S$, from which $i$ was removed.  Count how many available intermediate sets $S$ is clockwise from $p_0$ and call this number $h$. Then add a step to the walk $m$ with label $h_x$ where $x$ corresponds to whether the element was labeled left or right. 
    
    The step is placed so that it begins at height $h-1$ where $h$ is the number of available intermediate sets in $[\pi']$ after $i$ is removed. The slope of the new step is determined as follows. 
        \begin{enumerate}
            \item If $i \not \in P$, the step will be a horizontal step.
            \item If $i\in P$ and $i$ has a left label, the step will be a down step.
            \item If $i\in P$ and $i$ has a right label, the step will be an up step.
        \end{enumerate}
        
    \item Label $[\sigma]$ as $[\pi']$ and continue with the above process until $[\pi']$ consists of only $p_0$. The resulting walk is the final output of the map.
\end{enumerate}

We must prove that this map is well-defined. 
\begin{lem}
At every step of the process, the pinnacles of $[\pi']$ will always be in $P'$, regardless of how many elements have been removed.
\end{lem}
\begin{proof}
To see this, note that $[\pi']$ starts with all pinnacles in $P'$ by assumption, and no non-pinnacle can ever become a pinnacle since we always remove the smallest element.
\end{proof}

\begin{lem}
\label{remove element g}
At every step where $i$ is removed from $[\pi']$ to form $[\sigma]$, the intermediate set in $[\sigma]$ from which $i$ was removed will be available.
\end{lem}
\begin{proof}
Suppose the intermediate set in $[\sigma]$ from which $i$ was removed was unavailable in $[\sigma]$, which would force one of the two bounding elements in $P'$ to have a left label. Now any element $p \in P'$ with a left label is a cyclic vale at the start of the algorithm, and since there are no other elements in $P'$ between $p$ and $i$ to be pinnacles, it must be that $i>p$. But this contradicts the assumption that the smallest element is being removed. Therefore, $i$ must have been removed from an available intermediate set.
\end{proof}

\begin{lem}
\label{Lemma for g}
The algorithm $g$ will produce a connected walk.
\end{lem}
\begin{proof}
Consider any step, $m_i$ formed by removing element $i$ from $[\pi']$ and let $m_{i-1}$ be the step placed just before that (when $i-1$ was removed from $[\pi']$). Then by definition of $g$, if $r_i$ and $r_{i-1}$ are the starting heights of steps $m_i$ and $m_{i-1}$ respectively, we have that $r_i +1$ must be the number of available intermediate sets in $[\pi']$ right after $i$ was removed while $r_{i-1} + 1$ was the number of available intermediate sets in $[\pi']$ right before $i$ was removed. We now have three cases.

Case 1: Suppose that $i$ was not in $P'$. Then the number of available intermediate sets will be unchanged once $i$ is removed, and since $g$ will cause $m_i$ to be horizontal in this case, we have that $m_i$ will end at height $r_{i-1},$ thereby connecting it to the previous step.

Case 2: Suppose that $i\in P'$ and $i$ had a left label. Then the intermediate set it occupied will become available by Lemma \ref{remove element g} and there will be a net increase of one available intermediate set. Since in this case $g$ causes $m_i$ to have downward slope, we have that $m_i$ will end at height $r_{i-1},$ thereby connecting it to the previous step.

Case 3: Suppose that $i \in P'$ and $i$ had a right label. Then the intermediate sets on either side of $i$ must be available or else Lemma \ref{remove element g} will be contradicted when $i$ is removed. Removing $i$ will then join these available intermediate sets into one available set, and there will be a loss of one available intermediate set. Since in this case $g$ causes $m_i$ to have upward slope, we have that $m_i$ will end at height $r_{i-1},$ thereby connecting it to the previous step.

Therefore, since in every case we have that the newly placed step connects to the previous, we will have that the final walk will be connected. 
\end{proof}

\begin{prop}
The walk produced by the algorithm is in $M_n(P)$.
\end{prop}
\begin{proof}
By definition, $g$ causes every step $i$ to have its label in $[r_i+1]$ because the numeric label for $i$ is chosen from the available intermediate sets after it is removed, which is $r_i+1$. Also, $g$ forces all slanted steps to have the required left/right label by definition. Additionally, the walk will be connected by Lemma \ref{Lemma for g} and also begin on the $x$-axis because there will be only one available intermediate set when the process is finished (because at the end, all that is left of the permutation is $[p_0]$). Finally, every step of the walk must start weakly above the $x$-axis because the only time there can be no available intermediate sets in $[\pi']$ is at the start of the algorithm, since removing an element always leaves behind an available intermediate set by Lemma \ref{remove element g}. This means that $r_i$ must be at least zero for all $i$ except possibly the last one, which forces all steps to begin weakly above the $x$-axis.

Therefore, the result $m\in M_n(P)$ and we have that the map is well-defined. 
\end{proof}

\begin{prop}
The maps $f$ and $g$ defined above are inverses.
\end{prop}
\begin{proof}

First, given a walk $m\in M_n(P),$ we let $m_i$ be the step numbered $i$, that is, the $i$th step from the right in the original $m$ (note that if segments are deleted from $m$, we do not re-index or renumber the remaining segments). Let $f_i$ be the step in $f$ that deletes walk segment $m_i$ and adds element $i$ to $[\pi']$, and let $g_i$ be the step in $g$ that adds walk segment $m_i$ and removes element $i$ from $[\pi]$. Then we have that $f = f_1 \circ f_2 \circ \cdots \circ f_n$ and $g = g_n \circ g_{n-1} \circ \cdots \circ g_1$ by definition. 

For a given walk $m\in M_n(P)$ we define $S_i$ to be the state of the walk and permutation $[\pi']$ after applying $f_{i+1} \cdots \circ f_{n-1} \circ f_n$. Under this notation, we have that $S_n$ is the initial state consisting of the original walk $m$ and $[\pi'] = [n+1]$ and that $S_0$ is the final state consisting of the empty walk and the final permutation. We can then say that when applying $f$ to a walk, each $f_i$ will convert state $S_i$ into state $S_{i-1}$. Therefore, if we can show that $g_i$ applied to state $S_{i-1}$ gives state $S_i$, we will have that $g \circ f$ is the identity because the $g_i$ will move back through all the states step by step. 

Consider the walk and permutation at some state $S_{i}$. If we apply $f_i$ to this state, we will delete step $m_i$ with label $j_x$ and add element $i$ to $[\pi']$. Then $j$ will be the number of the intermediate set into which $i$ is inserted and the left/right label of $i$ in $\pi$ will be $x$. This will give us state $S_{i-1}.$ If we then apply $g_i$, we will remove $i$ and add step $m_i$ to walk $m$ with label $j_x$. 

Now we must show that the height and slope of $m_i$ are as they were in $S_i$. As was proven already in Lemma \ref{Lemma for f} the height of the start of step $m_i$ in $S_i$ was $h-1$ where $h$ is the number of available intermediate sets in $[\pi']$ in state $S_i$. But by definition of $g$ acting on state $S_{i-1}$, the height of the start of the step that $g$ places must be $h-1$ again since the number of intermediate sets after $i$ is removed from $[\pi']$ is the same as it was in $S_i$. Therefore, $m_i$ will be placed at the same height. To show that slope is preserved, we consider each of the three cases. If $m_i$ started as horizontal in $S_i$, then $i \not \in P$ because of the restrictions imposed on $M_n(P)$ and when we apply $g_i$ to state $S_{i-1}$ we will get a horizontal step back. If $m_i$ was an up step, we have by the restrictions imposed on $M_n(P)$ that $m_i\in P$ and that $m_i$ has a right label. Since the left/right label for $n$ in $[\pi']$ is the same as that for $m_i$, we have that $g_i$ will make $m_i$ an up step again. A similar proof shows that if $m_i$ was a down step before it was removed, it will be a down step again after it is returned. Therefore, $g_i$ undoes $f_i$.

Now for a given permutation $[\pi'] \in V_n(P)$ we define $S'_i$ to be the state of the walk and permutation $[\pi']$ after applying $g_{i} \cdots \circ g_{2} \circ g_1$. Under this notation, we have that $S'_n$ is the final state consisting of the final walk and permutation $[\pi'] = [n+1]$, and that $S'_0$ is the initial state consisting of the empty walk and the initial permutation. We can then say that when applying $g$ to a permutation, each $g_i$ will convert state $S'_{i-1}$ into state $S'_{i}$. Therefore, if we can show that $f_i$ applied to state $S'_{i}$ gives state $S'_{i-1}$, we will have that $f \circ g$ is the identity because the $f_i$ will move back through all the states step by step.

Consider the walk and permutation at some state $S_{i-1}$ and suppose we apply $g_i$ where we remove element $i$ with label $x$ from $[\pi']$ and add step $m_i$ to the walk. Then we will be in state $S'_i$ where the numeric label of $m_i$ will indicate the intermediate set, $I$, from which $i$ was removed and the left/right label will be $x$. If we then apply $f_i$, we will add $i$ back to the intermediate set corresponding to the label of $m_i$ with left/right label $x$, and so we have that $i$ will arrive back in $[\pi']$ in $I$ and have the same label as in state $S'_{i-1}$. 

We must now show that the position of $i$ in $[\pi']$ is as it was in $S'_{i-1}$. If $i$ was the only element in $I$ in state $S'_{i-1}$, it will arrive back in the same position in $[\pi']$ and we are done. If $i$ was not the only element, we note that since it was a vale and since all non-pinnacles in $I$ must form a decreasing sequence to the left of $i$ and an increasing sequence to the right of $i$ (or become pinnacles themselves), we have that $i$ was initially positioned next to the second smallest element in $I$. By the labeling conventions imposed on $[\pi']$, any element in $I$ to the left of $i$ will have a right label and any elements in $S$ to its right will have a left label. Therefore, no matter which of these two elements becomes the new vale after $i$ is removed, we are guaranteed that $f_i$ will put $i$ adjacent to this element again, and that it will be on the left if the element is labeled left and the right if it is labeled right. Therefore, $f_i$ undoes $g_i$.

Since $f$ and $g$ undo each other step by step, we have that they are inverses.
\end{proof}

Since the maps are inverses, we have that the sets are in bijection which finishes the proof of Theorem \ref{Main result 1}. We finish with a few remarks.

\begin{cor}
If we define $W_k$ to be the set of $y$ coordinate sequences of all walks using steps $[1,1]$ and $[1,-1]$ and ending at any height, then 
$$\sum_{Q \subseteq P} 2^{|Q|+1}|p_n(Q)| = 2^{n-k} \sum_{w\in W_k} \prod_{m=0}^k(w_m+1)^{p_m-p_{m+1}}$$

where we sum over $W_k$ instead of $R_k$.
\end{cor}
\begin{proof}
To see this, we note that if any walk $w\in W_k$ has a step that starts below the $x$-axis, then for some $i$ there will exist a $y$-coordinate of the walk, $w_i = -1$, for which $i \neq k$. This will cause $w_i+1 = 0$ and $p_i>p_{i+1}$ so that the weight of walk $w$ zeros out. Therefore, the sum reduces to summing over $R_k.$
\end{proof}

As a final comment, we note that if we require $P$ to be an admissible pinnacle set, Theorem \ref{Main result 1} reduces to Theorem 1.2 in \cite{Fang_2021}. This is because if $P$ is admissible, we know that $1 \not \in P$ which means the rightmost step of every walk in $M_n(P)$ will be horizontal. This prevents all walks from dipping below the $x$-axis, and we can safety restrict our definition of $R_k$ in Theorem \ref{Main result 1} to require that all $r_k$ are non-negative, which is the same set that Fang used. Dividing both sides of Theorem \ref{Main result 1} by two then gives Fang's formula.


\section{Recursion for Admissible Orderings}

Let $\omega$ be any ordering of the elements in some set $S \subseteq [n]$. Then for any permutation $\pi$ of a set of positive integers containing $S$, we say that $\ord (\pi) = \omega$ if the elements in $S$ appear in the same relative order in $\pi$ as they do in $\omega.$ We also wish to define the notion of standardizing to a set $S$. Given some permutation $\pi$ of distinct integers and a set $S$ of distinct integers such that $|S| = |\pi|$, we may use the unique order preserving map to take the elements of $\pi$ into $S$ to get a new permutation $\pi'$ with elements in $S$. In this case, we call $\pi'$ the standardization of $\pi$ to $S$. Note that if $S = [n]$ for $n$ the length of $\pi$, then this notion of standardizing coincides with the traditional one.

Now consider some $P = \{p_1 < p_2 < \cdots < p_k\} \subseteq [n]$ which need not be a pinnacle set (note the different indexing from the previous section). For the following theorems, we define an \emph{admissible ordering} of $P$ to be any ordering $\omega$ of the elements of $P$ such that there exists a $\pi \in S_n$ with $\Pin(\pi) = P$ and $\ord (\pi) = \omega$. If such a $\pi$ exists, we call it a \textit{witness} to the ordering. Under this definition, it makes sense to talk about the admissible orderings even when $P$ is not an admissible pinnacle set, since there will just be zero orderings. We will also define $N = \{n_1 <n_2< \ldots < n_{k+1}\}$ to be the set of the first $k+1$ positive integers that are not in $P$.

\begin{theorem}
Suppose $P = \{p_1<p_2< \ldots < p_k\}\subseteq [n]$ is an arbitrary subset of $[n]$. Let $N = \{n_1 <n_2< \ldots < n_{k+1}\}$ be the set of the first $k+1$ positive integers that are not in $P$. Then given a fixed ordering $\omega$ of $P$, there exists a $\pi\in S_n$ with $\Pin(\pi) = P$ and $\ord (\pi) = \omega$ 
if and only if there exists a permutation $\tau$ of the elements $P \cup N$ with $\Pin(\tau) = P$ and $\ord (\tau)  = \omega$.
\end{theorem}

\begin{proof}
Let $\pi \in S_n$ be a witness to some admissible ordering of $P$. We will build up a permutation $\tau$ on the elements of $P \cup N$ which has $\Pin(\tau) = P$. An example of this process follows the proof. 

To start, let $\tau = \pi$ and identify each of the $k+1$ blocks of consecutive non-pinnacles in $\tau$. For each block, delete all but the smallest non-pinnacle. Finally, take the set of all $k+1$ remaining non-pinnacles in the result and standardize them to the set $N$ to get permutation $\tau$. Then $\tau$ will be a permutation of the elements of $P \cup N$ by definition with elements of $P$ in order $\omega$. We must also have that $\Pin (\tau) = P$ since every $p \in P$ started as a pinnacle in $\pi,$ and the elements on either side of $p$ could only decrease throughout the steps of this process. Since all other elements end up adjacent to a pinnacle, there can be no added pinnacles, and so we have $\tau$ as specified. 

To see the reverse direction, simply take any $\tau$ with a given ordering of the elements of $P$ as the pinnacle set and add the elements $[n]\setminus (P \cup N)$ in increasing order to the end of $\tau$. The result will be in $S_n$ and have pinnacle set $P$ in the same order as in $\tau$.
\end{proof}

As an example, suppose that $n=8, P = \{3,7\}$ and $\pi = 5\textbf{7}642\textbf{3}18.$
Then $N = \{1,2,4\}, \, \omega = 73,$ and we build up $\tau$ in the following steps:

$$\tau = 5\textbf{7}642\textbf{3}18$$
$$\tau = 5\textbf{7}2\textbf{3}1$$
$$\tau = 4\textbf{7}2\textbf{3}1$$

We may now state a corollary which will always allow us to reduce to the case where we only need to look for witnesses in $S_{2k+1}$

\begin{cor}
Given $P = \{p_1 < p_2 < \cdots < p_k\}$ and $N$ as above, standardize the set $P \cup N$ and let $P'$ and $N'$ be the sets that $P$ and $N$ standardize to, respectively. Then the number of admissible orderings of set $P$ for $\pi \in S_n$ is the same as the number of admissible orderings of set $P'$ for $\pi \in S_{2k+1}$.
\end{cor}
\begin{proof}
Simply standardize the permutation $\tau$ given in the proof of the previous theorem.
\end{proof}

For example, if we standardize the the $\tau = 4\textbf{7}2\textbf{3}1$ we had at the end of the previous example, we would get $\tau = 4\textbf{5}2\textbf{3}1 \in S_5$ where now $P' = \{5,3\}$. 

In all that follows we will assume that 
\begin{equation}\label{Condition 1}
  P=\{p_1<p_2<\ldots<p_k\} \text{ and } n = 2k+1  
\end{equation}
since counting admissible orders when $n$ is larger can be reduced to this case. Note that because of this, specifying $P$ is enough to determine $N$ since $N = [2k + 1] \setminus P$.

Let $O(P)$ be the set of admissible orderings of set $P$, and $o(P) = |O(P)|$. More generally, if we let $N_i = \{n_1<n_2< \ldots < n_i\}$ be the $i$ smallest elements of $N$, we define 
$$
O_i(P) = \{\omega \in S_{P \cup N_i} \mid 
\text{there is $\pi \in S_{2k+1}$  with $\ord (\pi) = \omega$  and  $\Pin(\pi) = P$}\}
$$ 
where $S_{A}$ is the set of all permutations of the set $A$. We further define $o_i(P) = |O_i(P)|$ and note that $O_0(P) = O(P).$

Essentially, $O_i(P)$ ``keeps track'' of the positions of not only the elements of $P$, but of the positions of the $i$ smallest elements of $N$, too. As an example, consider the set $P = \{3,5\}$. Then $o(P) = 2$ since all orderings are admissible. However, $o_1(P) = 4$ since 
$$O_1(P) = \{135, 315, 513, 531\}$$
with corresponding witnesses
$$\{1\mathbf{3}2\mathbf{5}4, 2\mathbf{3}1\mathbf{5}4, 4\mathbf{5}1\mathbf{3}2, 4\mathbf{5}2\mathbf{3}1\}.$$
Note that the orderings $351$ and $153$ are not included however because there is no way to insert the elements $2$ and $4$ to get a permutation in $S_5$ having pinnacle set $P$.

In the following results, we will need to talk about witnesses to specific orderings in $O_i(P)$. The following lemma will allow us to make certain assumptions about the witness we wish to pick.

\begin{lem}\label{swap small elements}
Consider any $\omega \in O_i(P)$ where the conditions in (\ref{Condition 1}) hold and where $i<p_1$. Then if $\omega'$ is obtained by permuting the elements in $[i]$ in $\omega$ among their indices, then $\omega' \in O_i(P)$ also.
\end{lem}
\begin{proof}
Since $\omega$ is an admissible ordering, there must exist at least one witness to it which we will call $\pi$. Because $n = 2k+1$ we know that in $\pi$ the pinnacles and non-pinnacles, which are the elements in $P$ and not in $P$ respectively, alternate beginning and ending with a non-pinnacle. Therefore each element of $[i]$ will be adjacent to two pinnacles larger than itself, and since all values in $[i]$ are smaller than all pinnacles, this will still be the case after permuting the elements in $[i]$. Therefore, no elements in $[i]$ can become pinnacles. Additionally, no pinnacles can be lost because the only elements to move are those in $[i]$ and if an element of $[i]$ adjacent to a pinnacle gets swapped, it will just be replaced by a different value in $[i]$ that is still smaller than the pinnacle. Finally, we have that any element $j$ that is not in $P$ or $[i]$ cannot become a pinnacle because neither $j$ nor the two elements of $P$ initially adjacent to it will move during the swap. Therefore, no pinnacles are gained or lost, and so the resulting permutation will be a witness to the corresponding ordering $\omega'$, which shows it must be in $O_i(P)$.
\end{proof}

Given $P \neq\emptyset$ and corresponding $N$, consider the set $[2k+1]\setminus \{p_1, n_1\}$ and standardize so that the elements are in $[2k-1]$. Let $P'$ be the set that the elements in $P\setminus \{p_1\}$ get mapped to which will always be $\{p_2-2,p_3-2, \ldots p_k-2\}$ if $i < p_1$. Then define the \emph{reduction operator} $r(P) = P'$. 

We now wish to form a bijection between a subset of the elements in $O_i(P)$ and those in some $O_{i'}(P')$. 
To this end, for $j \in \{0,1,2\}$ we define 
$$O_i^{j}(P) = \{\omega \in O_i(P) \mid \text{$p_1$ is adjacent to $j$ elements in $[i]$} \}$$ and let $o_i^{j}(P) = |O_i^{j}(P)|$. 

We then have the following result.

\begin{lem} \label{ordering map}
Suppose $i \geq 0$, that the conditions in (\ref{Condition 1}) hold, and that $P$ is non-empty. Then if for some $j \in \{0,1,2\}$ we have that $i<p_1$ and $O_i^j(P)\neq\emptyset$ and we let $P' = r(P)$, then we have the following 
$$o_i^{j}(P) = 
 \begin{cases} 
      i(i-1) o_{i-1}(P')     &\text{if } j = 2,\\
       2i o_i(P')            &\text{if }  j = 1,\\
      o_{i+1}(P')            &\text{if }  j = 0.
\end{cases}
$$
\end{lem}
\begin{proof}
Note that since we have that $O_i^j(P)\neq\emptyset$, there must exist at least $j$ elements in $[i]$ and $2-j$ elements not in $[i]$ but less than $p_1$ to surround $p_1$. Therefore, we have the inequalities $i\geq j$ and $i+2-j < p_1$.

We will consider the case where $j=2$ in detail and note the changes that must occur for when $j=0,1$. Note first that by Lemma \ref{swap small elements} we have that the existence of an ordering $\omega \in O_i^{2}(P)$ in which $\omega$ contains the factor $xp_1y$ implies the existence of one containing the factor $1p_12$ and visa versa. If we let 
$$A = \{\omega\in O_i^2(P) \mid \text{$1p_12$ is a factor of $\omega$}\}$$
then we will have that 
$$
o_i^2(P) = i(i-1)\cdot |A|.
$$
where we have the factor of $i(i-1)$ because there are $i$ ways to choose an element in $[i]$ to the left of $p_1$ and $i-1$ ways to choose a a different element in $[i]$ to the right. We may now define a map between set $A$ and $O_{i-1}(P')$ in which we remove the factor $1p_1$ from $\omega$ and standardize to the set $P'\cup [i-1]$. To show this ordering has a witness, first consider any witness, $\pi$, to $\omega$ and note that $\pi$ must also contain the factor $1p_12$ since elements in $P$ must alternate with those that do not, and so no new elements can be inserted between $1,p_1,$ and $2$. If we then remove $1p_1$ from $\pi$ and standardize to get $\pi'$, we are left with a witness to the ordering $\omega'$. Note that the element to the left of $1$ in $\pi$, if it exists, must have originally a pinnacle, and since in $\pi'$ it will end up being adjacent to $1$ again, it must remain a pinnacle. Furthermore, since $2$ cannot become a pinnacle, and since the relative order of all other elements was preserved, we have that the pinnacle set of $\pi'$ is $P'$ which shows that $\omega'$ is admissible. 

To reverse this map, consider some $\omega' \in O_{i-1}(P')$ and standardize to the set $(P \cup [i]) \setminus \{p_1, 1\}$. Note that since $i\geq j=2$, 
we have that $\omega'$ will contain the element $1$, which will standardize to $2$. Replace $2$ with $1p_12$ to get the ordering $\omega$ which will be in $A$ if we can show it has a witness. To show this, consider any witness $\pi'$ to $\omega'$, standardize to $[2k+1]\setminus \{p_1, 1\}$, and replace $2$ with $1p_12$ to get $\pi$, which will be a witness to $\omega$. An argument similar to the one above shows that the pinnacle set will be $P$, and so we have that $\omega$ is admissible, completing the proof in this case. 

Since these maps are clearly inverses, we have that $|A| = o_{i-1}(P')$ which completes the case when $j=2$.

For the case where $j=1$, we can use a similar proof where we use Lemma \ref{swap small elements} to show that
$$
o_i^2(P) = 2i\cdot |\{\omega\in O_i^2(P) \mid \text{$1p_1$ is a factor of $\omega$}\}|.
$$
and then define a bijection to simply delete $p_1$ from $\omega$ and standardize to $P' \cup [i]$ to get $\omega'$. Care must be taken to prove this results in an ordering in $O_i(P')$. To prove it has has a witness, one can use the fact that $i+2-j<p_1$ to argue for the existence of a witness $\pi \in S_{2k+1}$ to $\omega$ which contains the factor $1p_1(p_1-1)$ where $p_1-1 \not \in [i]$. Replacing this factor with $1$ and standardizing will result in a witness for $\omega'$ with pinnacle set $P'$. The fact that $i+2-j<p_1$, together with the fact that $i \geq j$, must be used again when showing that the reverse of this map is well-defined. 

For the case where $j=0$, we may define our bijection between $O_i^j(P)$ and $O_{i+1}(P')$ by taking some $\omega\in O_i^j(P)$, replacing $p_1$ with $0$, and standardizing to $P' \cup [i+1]$ to get $\omega'$. To prove that $\omega$ has a witness, one can use the fact that $i+2-j<p_1$ to argue for the existence of a witness $\pi \in S_{2k+1}$ to $\omega$ which contains the factor $(p_1-2)p_1(p_1-1)$ where $p_1-1, p_1-2 \not \in [i]$. Replacing this factor with $0$ and standardizing will then result in a witness for $\omega'$, which can be shown to be in $O_{i+1}(P')$. The fact that $i+2-j<p_1$ must be used again when showing that the reverse of this map is well-defined. 
\end{proof}

Given some set $P$, we make the following definition:
$$\delta_j = 
 \begin{cases} 
      1           &\text{if } p_1>j,\\
      0           &\text{otherwise}.
\end{cases}
$$
Then we can state our main result.

\begin{theorem}
Let $|P|\geq 1$, $P' = r(P)$, $0\leq i < p_1$, and suppose the conditions in (\ref{Condition 1}) hold. Then we have the following recursion 
$$o_i(P) = i(i-1)o_{i-1}(P') + (2i) \delta_{i+1} o_i(P') + \delta_{i+2} o_{i+1}(P').$$
\end{theorem}
\begin{proof}
Note that $O_i(P)$ can be partitioned into three (possibly empty) subsets depending on whether $p_1$ is directly adjacent to either two, one, or no elements of $N_i$. 
It follows that 
$$
o_i(P)=o_i^2(P)+o_i^1(P)+o_i^0(P).
$$
Therefore, we will have our result if we can show that 
$$o_i^{j}(P) = 
 \begin{cases} 
      i(i-1) o_{i-1}(P')                &\text{if } j = 2,\\
      2i \delta_{i+1} o_i(P')           &\text{if } j = 1,\\
      \delta_{i+2} o_{i+1}(P')          &\text{if } j = 0.
\end{cases}
$$
If $O_i^j(P)$ is non-empty, then we have our result by Lemma \ref{ordering map} since we will have that $i+2-j<p_1$ and  the corresponding delta is $1$. Therefore, we only need to concern ourselves with the case where $o_i^j(P) = 0$, in which we must prove that the corresponding expression for $o_i^j(P)$ is also zero.

First, we claim that $o_i^j(P) = 0$ if and only if either $P$ is not admissible, $i+2-j\geq p_1$, or $j>i$. That each of these three conditions independently forces $o_i^j(P) = 0$ is not hard to see since we will either have that $o_i(P) = 0$ or that there are not enough of the necessary elements to surround $p_1$ and have it still be a pinnacle. For the reverse direction, we prove the contrapositive. Suppose that $P$ is admissible, that $i+2-j< p_1$, and that $j\leq i$. In this case, chose any witness to any ordering of $P$ and call it $\pi$. Now by our two inequalities, we know there must exist two elements, $x$ and $y,$ in $\pi$ such that $j$ of them are in $[i]$ and $2-j$ of them are between $i$ and $p_1$. If we then swap $x$ and $y$ with the pair of elements to either side of $p_1$ in $\pi$, we will have a witness to an ordering in $O_i^j(P)$. The fact that the pinnacle set of this witness is indeed $P$ can be proven similarly to Lemma \ref{swap small elements} since $\pi$ alternated pinnacles and non-pinnacles, and since all elements swapped were smaller than all pinnacles. Therefore, we have an ordering in $O_i^j(P)$, which proves it is non-empty. This finishes the proof of the claim.

Therefore, in the case where $o_i^j(P) = 0$ we may suppose that at least one of these three conditions is true. If $j>i$, we immediately have that all three expressions for $o_i^j(P)$ are zero and so we still have equality even in the case when $j>i$. If we assume that $i+2-j\geq p_1$ for some $j \in \{0,1,2\}$, we must have that $j<2$ since otherwise we contradict our assumption that $i< p_1$. But if $j=0,1$, we immediately have that $\delta_{i+2-j}$ is zero and so again we have equality. 

To finish the proof, all we have left to show is that all three expressions are zero if $j\leq i$, $i+2-j< p_1$, but $P$ is not admissible. Clearly if $P'$ fails to be admissible too, all three terms will immediately be zero, so suppose that $P'$ is admissible. This is equivalent to saying that for every element $p \in P'$, the number of elements in $P'$ less than $p$ is exceeded by the number of elements not in $P'$ less than $p$. This will still be the case for all corresponding pinnacles in $P$ since both elements added are smaller than all other pinnacles, and so if $P$ fails to be admissible, it must be because the added element of $P$, $p_1$, is less than $3$. But supposing $j\leq i$ and $i+2-j< p_1$ together force $p_1>2$, and so we have a contradiction. 

Therefore, since in all cases where $o_i^j(P) = 0$ we have that the corresponding expression is also zero, we have our result. 
\end{proof}

\begin{cor}
This recursion can be used to compute $o(P) = o_0(P)$. Furthermore, the number of terms that need to be computed will be asymptotically equal to $k^2/4$.
\end{cor}
\begin{proof}
First, note that every step reduces the size of $P$ by one, and so if we induct on $k$, we have the base case $P = \emptyset$, and $|N| = 1$ which always has only one ordering. In order to show that we may repeatedly use this recursion however, we must show that we never need to compute some $o_i(P)$ where $i\geq p_1$. 

First, note that on the first step when $i=0$, we must have that $i<p_1$ since $P \neq \emptyset$. Now consider any application of the recursion on $o_i(P)$ where $i<p_1$. Since $p'_1 = p_2-2$ forces $p'_1 \geq p_1-1$ by definition of $r(P)$, we know that $p'_1>i-1$ and so the recursion may be applied on $o_{i-1}(P')$. Now we consider $o_{i}(P')$. If $i < p'_1$ we may apply the recursion; if not, then we will have that $i \geq p_1-1$ which means $p_1\leq i+1$. But then we will have that $\delta_{i+1} = 0$, and so we will not have to calculate $o_{i}(P')$. Finally, we consider $o_{i+1}(P')$. If $i+1 < p'_1$ we may apply the recursion; if not, then we will have that $i+1 \geq p_1-1$ and so $p_1\leq i+2$. But then we will have that $\delta_{i+2} = 0$, and so we will not have to calculate $o_{i+1}(P')$. Therefore, the recursion never generates a term that cannot then be further broken down.

As to the number of terms that need to be calculated, let $P^j = r^j(P)$ applying the $r$ operator $j$ times where $0\leq j \leq k$. Note that since $i$ starts out as $0$ and $i$ can only increase by at most $1$ each step, we must have that $i\leq j$ for any term $o_i(P^j)$. Additionally, since all permutations have elements in $[2(k-j)+1]$, we know that $i \leq k-j+1$ since the remaining $k-j$ elements are in $P^j$. Therefore when calculating terms of the form $o_i(P^j)$, if $j\leq k/2$ we know that $0\leq i \leq j$ while if $k/2 \leq j \leq k$ we know that $0\leq i \leq k-j+1$. Therefore, if we sum up all the different values that $i$ can take over all values of $j$, we have that the total number of terms is at most
$$
\sum_{j=0}^{\floor{k/2}} (j+1) + \sum_{j=\ceil{(k+1)/2}}^k (k-j+2) = \sum_{j=0}^{\floor{k/2}} (j+1) + \sum^{\floor{k/2}}_{j = 0} (j+2) =  \binom{\floor{k/2}+2}{2} + \binom{\floor{k/2}+2}{2} + \floor{k/2}+1
$$
which is asymptotically $k^2/4$.
\end{proof}

\section{Block representation of pinnacle sets}

Given a (not necessarily admissible) set $P = \{p_1 < p_2 < \cdots < p_k\} \subseteq [n]$, we may choose to represent $P$ as an array of some length $n \geq p_k$ consisting of 0s and 1s such that a given index $i$ contains a 1 if and only if $i \in P$. For example, if we had 
$$P = \{3,6,7\}$$
then the corresponding array for $n = 7$ would be
$$[0,0,1,0,0,1,1].$$
For some fixed $n$, we can clearly we can go back and forth between these representations without loss of information, and the array will correspond to an admissible pinnacle set if and only if it is a ballot sequence, that is, at every index the number of 0s exceeds the number of 1s occurring before that index (we will not use this fact, but it is interesting). Now for the purposes of our proofs, it will be useful to think of $P$ using the array notation because it allows use to visually divide $P$ up into ``blocks,'' which is to say, we may break the array up into sub-arrays. For instance, in our previous example if we let 
$$B_1 = [0,0,1],\, B_2 = [0,0,1,1]$$
then we may say that $P$ can be written as $B_1B_2$ where we concatenate the arrays to recover $P$. We may then apply algorithms recursively to each block, and combine the results to say something about $P$. For our purposes, it won't matter how we break up $P$ provided that all blocks are non-empty, and so we will not (at this point) structure our notation to indicate how much of $P$ is contained in $B_1$.

As an aside (but something that will not be used in the proof) we may choose our blocks carefully so that they correspond to trees. This is done by simply looking at the binary tree representation of $P$ (which is given in greater detail below), and letting $B_1$ consist of those indices appearing in the first tree, $B_2$ of those indices appearing in the second tree, and so on. 
Our approach is a little more general however as it allows us use blocks that contain ``part'' of a tree, which will be useful in our applications. Using this particular division into blocks however will allow us to say some things about the tree representation of $P$ later. 

We now must introduce some new ideas and notation.

Given a block $B$, which is to say, an array of some length $m$ containing only 0s and 1s, we may define a set, $P_B$, to contain all those indices in $B$ that are 1s. We then define $P(B)$ to be the set of permutations in $S_m$ with pinnacle set $P_B$, and further define $p(B) =|P(B)|$. For instance, suppose that in our running example we took $B_1 = [0,0,1]$. Then $P_{B_1} = \{3\}$ and $p(B_1) = 2$ since there are two permutations in $S_3$ that have pinnacle set $\{3\}$. On the other hand, if we took $B_2 = [0,0,1,1]$ we would have $P_{B_2} = \{3,4\}$ and $p(B_2) = 0$ since $P_{B_2}$ is not an admissible pinnacle set.

We now take this a step further. Given a block $B$ of length $n$ with corresponding pinnacle set $P_B$, it will be desirable to modify $P_B$ by adding additional elements larger than $n$ (that is to say, we wish to force additional pinnacles). Furthermore, we will want to require that these added elements are indistinguishable from one another, which is a deviation from normal pinnacle sets where all pinnacles are different. Pinnacle sets are still defined however even when considering permutations with repeated numbers, and so the notion still makes sense.

More formally, define $P(B)^i$ to be all permutations of the multiset $\{1,2,\ldots, m\} \cup \{(n+1)^i\}$ that have pinnacle multiset $P_B \cup \{(n+1)^i\}$, and define $p(B)^i = |P(B)^i|$. For instance, if $B = [0,0,1,0,0,0]$, then $P_B = \{3\}$ and $m=6$. To compute $P(B)^2$, we must ask which permutations of the set $\{1,2,3,4,5,6,7,7\}$ have pinnacle set $\{3,7,7\}$, of which there are 36 (there are 3 ways to decide the order of $3,7,7$ within the permutation, 2 ways to decide the order of 1 and 2 on either side of 3, and 6 ways to arrange the elements 4,5,6 to fill the remaining gaps between the pinnacles). On the other hand, $p(B)^5 = 0$ because there would be too many pinnacles forced. 

There is a similar idea for forcing cyclic vales. Suppose once again that we have a block $B$ of length $n$ with corresponding pinnacle set $P_B$. It will be useful to talk about the set of permutations having not only pinnacle set $P_B$, but also a cyclic vale set that contains a certain set of desired elements. In fact, similar to how above we added new, large elements and forced them to be pinnacles, we will want to add new, small elements and force them to be cyclic vales. 

More formally, define $P(B)_j$ to be all permutations of the multiset $\{1,2,\ldots, n\} \cup \{0^j\}$ that have pinnacle set $P_B$ and cyclic vale set containing the multiset $\{0^j\}$, and define $p(B)_j = |P(B)_j|$ . For instance, if $B = [0,1,1]$, then $P_B = \{2,3\}$ and $m=3$. To compute $P(B)_2$, we must ask how many permutations of the set $\{0,0,1,2,3\}$ have pinnacle set $\{2,3\}$ and also cyclic vale set containing $\{0^2\}$, of which there are 6: there are 6 ways to order the elements $1,2,3$ and then we may always insert the two 0s in a unique way in the remaining positions around $2$ and $3$ to force them to be pinnacles. On the other hand, $P(B)_4 = 0$ because the only way to force all four zeros to be cyclic vales is if they alternate with the numbers $1,2,3$, but then this will force $1$ to become a pinnacle.

While the above ideas may seem strange, hopefully they will feel more natural in a minute. The gist is that in addition to requiring that a particular set becomes a pinnacle set, we may now throw in additional large elements and require them to be pinnacles, or additional small elements and require them to be cyclic vales. In fact, we will often do both at once, and to this end, define the notation $P(B)^i_j$ to be all permutations of the multiset $\{1,2,\ldots, n\}\cup\{0^j, (n+1)^i\}$ that have pinnacle multiset $P_B \cup \{(n+1)^i\}$ and vales set containing $\{0^j\}$. In the event that no such permutations exist, we simply say that $P(B)^i_j = \emptyset$. Note that if any such permutations do exist, it will always be the case that no two elements of $\{(n+1)^i\}$ are adjacent, and that no two elements of $\{0^j\}$ are adjacent, as that would not allow them to all be pinnacles or cyclic vales respectively. This fact will be helpful in our first proof.

Finally, we need to remind the reader of cyclic permutations. In general, any linear permutation in $\{1,2,\ldots n\} \cup\{0^j, (n+1)^i\}$ included in $P(B)^i_j$, that is, having pinnacle multiset $P_B \cup \{(n+1)^i\}$ and cyclic vale set containing $\{0^j\}$, corresponds to a cyclic permutation in $\{1,2,\ldots n, \infty\} \cup\{0^j, (n+1)^i\}$ having cyclic pinnacle multiset $P_B \cup \{(n+1)^i, \infty\}$ and cyclic vale set containing $\{0^j\}$. This correspondence is achieved simply by appending $\infty$ onto the beginning of the linear permutation, and wrapping it around into a cyclic permutation. The reverse operation is also well defined. Therefore, we define $cP(B)^i_j$ to be the set of cyclic permutations corresponding to the linear ones in $P(B)^i_j$. 

\begin{theorem}\label{Section 3 recursion}
Suppose that some block $B$ (as described above) of length $n$ is decomposed into two, non-empty parts $B_1$ and $B_2$ such that the concatenation $B_1B_2$ gives back $B$. Suppose further that the number of ones in $B_2$ is $b$. Then 
$$
p(B)^i_j = \sum_{a=1}^{b+i+1} p(B_1)^{a-1}_j p(B_2)^i_{a} 
$$ 
\end{theorem}
\begin{proof}
First, let $L$ be the set of all elements corresponding to those in block $B_1$ together with the set $\{0^j\}$. Similarly let $H$ be the set of all remaining elements in $[n]$ not in $L$ together with the set $\{(n+1)^i, \infty\}$. Then $L$ will be the ``low'' set and $H$ the ``high'' set in the sense that all elements in $L$ will be smaller than all elements in $H$, and given any $[\pi] \in cP(B)^i_j$, we have that all elements in $[\pi]$ will be in either $L$ or $H$. Since $p(B)^i_j = |cP(B)^i_j|$ it suffices to consider only cyclic permutations to prove the result.

Therefore given some $[\pi] \in cP(B)^i_j$ we may decompose it into maximal consecutive sequences of elements either entirely in $H$ or entirely in $L$. Since such sequences from $H$ must alternate with those in $L$, we have that there are the same number of each and we call this number the \textit{alternations} of $[\pi]$ and denote it by $a$. Note that $a$ must be at least $1$ since we assumed both $B_1$ and $B_2$ were non-empty, and furthermore we know that $a$ cannot exceed the number of cyclic pinnacles in $H$ since every consecutive string of elements in $H$ is adjacent to two values in $L$ and so must contain at least one cyclic pinnacle of $[\pi]$. Since every cyclic pinnacle element from $H$ must either be one of the $b$ pinnacles from $B_2$, $n+1$, or $\infty$, we have an upper bound of $b+i+1$ for $a$.

Therefore, if we fix some $1\leq a \leq b+i+1$, we may reduce to proving that the number of permutations in $cP(B)^i_j$ with $a$ alternations is $|P(B_1)^{a-1}_j||P(B_2)^i_a|$ since the values of $a$ clearly partition all permutations in $cP(B)^i_j$.

We prove this by constructing a bijection. An example follows this proof. Let $l = |B_1|, h = |B_2|$, and suppose for the forward direction that $[\pi] \in cP(B)^{i}_j$ has $a$ alternations. Then in $[\pi]$ we may replace the block of consecutive elements in $H$ containing $\infty$ with the element $\infty$ and each of the other $a-1$ blocks of consecutive elements in $H$ with the element $l+1$ to get $[\pi_L]$. Additionally, we may take the original $[\pi]$ and replace each of the $a$ blocks of consecutive elements in $L$ with the element $0$ and subtract $l$ from all remaining elements to get the cyclic permutation $[\pi_H]$. We must now show that $[\pi_L] \in cP(B_1)^{a-1}_j$ and $[\pi_H] \in cP(B_2)^i_a$.

From the construction given, it is easy to see that $[\pi_L]$ is a permutation of the elements $\{1, \ldots, l\} \cup \{(l+1)^{a-1},\infty, 0^j\}$ and that $[\pi_H]$ is a permutation of the elements $\{1, \ldots, h\} \cup \{(h+1)^{i}, \infty, 0^a\}$. Therefore, we need only show that all the necessary cyclic pinnacles and vales are formed in both cases. 

Consider $[\pi_L]$. Clearly we must have that every element in $\{(l+1)^{a-1}\}$ is a cyclic pinnacle since they are larger than all elements in $\{0, 1,\ldots, l\}$, and since none of them are adjacent to each other or to $\infty$ because everything in $\{(l+1)^{a-1}, \infty\}$ originated from separate blocks of elements in $[\pi]$. Furthermore, given any element $x \in L$ the relative order between $x$ and the elements on either side is the same in $[\pi_L]$ as it was in $[\pi]$, since the only change was that an element larger than $x$ may have been replaced by another element larger than $x$. Therefore, since all elements in $P_{B_1}$ and $\{0^j\}$ were cyclic pinnacles and vales respectively in $[\pi]$, they will continue to be so in $[\pi_L]$, and so we are done. The proof for showing $[\pi_H] \in cP(B_2)^i_a$ is almost identical, and therefore omitted. 

Therefore, our map is well-defined. To invert it, suppose that $[\pi_L] \in cP(B_1)^{a-1}_j$ and $[\pi_H] \in cP(B_2)^i_a$. To combine them, first add $l$ to all elements of $[\pi_H]$ to get $[\pi'_{H}]$. Then write $[\pi_L]$ in the form $[\infty L_1 (l+1) L_2 (l+1) \cdots (l+1) L_a]$ where each $L_i$ is a maximal consecutive sequence of elements less than $l+1$. Similarly, we may write $\pi'_H$ in the form $[H_1 l H_2 l \cdots l H_a l]$ where each $H_i$ is a maximal consecutive sequence of elements greater than than $l$, and where $H_1$ contains $\infty$. We then construct the cyclic permutation $[\pi] = [H_1L_1H_2L_2\cdots H_aL_a].$ It is easy to verify that $[\pi]$ contains $j$ copies of the element $0$, $i$ copies of the element $n+1$, and that all other elements are either $\infty$ or in $B$. Furthermore we note that given any element $x$ in $[\pi]$, the relative sizes of $x$ and its two adjacent elements will be the same as they were for $x$ back in either $[\pi_L]$ or $[\pi_H]$. Therefore, all cyclic pinnacles and vales in $[\pi]$ are directly derived from those in $[\pi_L]$ or $[\pi_H]$. This means that all $0$'s will be vales, all elements of size $n+1$ will be cyclic pinnacles, and all other cyclic pinnacles will correspond to a cyclic pinnacle in either $B_L$ or $B_H$. Therefore, $[\pi] \in P(B)^i_j$.

Since these maps are clearly inverses step by step, we have our bijection, and therefore our result. 
\end{proof}

As an example of this process, consider $B = [00101001]$ where $B_1 = [001]$ and $B_2 = [01001]$. Then we have that $P_B = \{3,5,8\}$ while $P_{B_1} = \{3\}$ and $P_{B_2} = \{2,5\}$ which correspond to the elements $5,8$ in $P_B$ if you add $|B_1| = 3$. We may then consider the permutation $[\pi] = [\infty 0523081964097] \in cP(B)^2_3$. If we write in boldface those elements of $[\pi]$ that are in set $L$ defined in the proof, we have 
$$[\pi] = [\infty \05\2\3\08\1964\097]$$ 
and now it is easier to see that there are 4 alternations were we switch from an element of $L$ to one in $H$ and then back. Then we have 

$$[\pi_L] = [\infty 042304140]$$ and
$$[\pi_H] = [\infty 02050631064]$$
which we get by standardizing
$$[\infty 05080964097]$$

This recursion is very useful, as it allows us to give an efficient method for calculating $p_n(P) = |\{\pi \in S_n | \Pin \pi  = P\}|$ using $O(k^2\log n + k^4)$
arithmetic operations where $k = |P|$. It also allows us to derive results on the tree representation of pinnacles sets, which we do later.

For a pinnacle set represented by block $B$, we will call $B$ a \textit{segregated} block if it takes the form $B = [0^x, 1^y]$ in which we allow $x$ and $y$ to potentially be zero. That is to say, $B$ must consist of some number of $0$'s followed by some number of $1$'s with no other alternations. Our method to calculate $p_n(P)$ has two parts. First, we will present an efficient formula for calculating $p(B)^i_j$ where $B$ is a segregated block. We will then show that by using the recursion above, we may always write $P(A)$ for any block $A$ in terms of these numbers, thereby giving a formula for the generic case.

\begin{lem}
\label{Section 3 segregated lemma}
Let $B = [0^x,1^y]$ be a segregated block where $x+y = n$. Then if $S(n,m)$ is the Stirling number of the second kind, we have

$$p(B)^i_j = 2^{x+j-c}\ \frac{c!(c-1)!}{i!} \sum_{m=0}^j \frac{1}{m!} \binom{c-m}{j-m} S(x,c-m)$$
where $c = i+y+1$ is the number of cyclic pinnacles.
\end{lem}
\begin{proof}
Given $\pi \in P(B)^i_j$, we may break $\pi$ down 
into a sequence of pinnacles in the set $P = P_B \cup \{(n+1)^i\}$ and a sequence of intermediate permutations, where each intermediate permutation is defined to be a maximal set of consecutive elements in $\pi$ consisting of no pinnacles. Furthermore, we know that the pinnacles must alternate with the intermediate permutations so that $\pi$ takes the form $I_1p_1'I_2p_2'\cdots I_{y+i}p_{y+i}'I_{c}$ where $I_m$ are the intermediate permutations, $P = \{p_1',\ldots,p_{y+i}'\}$, and $c = i+y+1$ can be viewed as the number of cyclic pinnacles if we count $\infty$. Therefore, since every pinnacle is larger than every non-pinnacle, it will be the case that given any ordering of the pinnacles $P$ and any set of $c$ nonempty intermediate permutations, every ordering of those intermediate permutations, when interleaved with the pinnacle elements, will form a permutation in $P(B)^i_j$. There are $(c-1)!/i!$ ways to order the pinnacle elements since $i$ of them are indistinguishable, and so to get our result it is enough to show that the number of ways of generating and ordering $c$ non-empty intermediate permutations is

$$2^{x+j-c}\ c!\sum_{m=0}^j \frac{1}{m!} \binom{c-m}{j-m} S(x,c-m)$$

Note that since intermediate permutations cannot contain pinnacles, they must contain exactly one cyclic vale, which will always be the smallest element. Additionally, we must have that each intermediate set can contain at most one $0$, since containing more would force one of the $0$'s to not be a vale, which breaks the requirement of $P(B)^i_j$. Therefore, we may partition all possible sets of intermediate permutations by the number of permutations $m$ that contain only the element $0$. It then follows that each of the $c-m$ other intermediate permutations must contain at least one element of $[x]$, and will therefore be distinguishable. So for a given $m$ the number of distinct orderings of a given set of intermediate permutations is $c!/m!$, as specified in the formula.

Now given an intermediate permutation, let its intermediate set be the corresponding set of elements in the permutation, which will have no repeats since we have already shown that the element $0$ cannot appear twice. If an intermediate permutation has $l$ elements, then we may assign $l-1$ of them a left or right label, corresponding to which side of the smallest element they appear. Alternatively, if we start from an intermediate set and give all elements except the smallest a left/right label, then there will be exactly one way to construct an intermediate permutation from it. Namely, all left labeled elements will form a decreasing string to the left of the smallest element, and all the right labeled elements an increasing string to the right. Therefore, counting intermediate permutations can be reduced to counting intermediate sets and multiplying by a power of $2$. Since every collection of intermediate sets will have $x+j$ elements total, and since all of them except the $c$ cyclic vales will be given a label, this power of $2$ will be as stated in the above relation. 

Therefore, we have reduced to the case of counting the number of ways of partitioning the $x+i$ non-pinnacles into $c$ non-empty intermediate sets such that no set contains more than one $0$. This will force exactly $j$ of the intermediate sets to contain a $0$, and so all that remains is to place the final $x$ elements in the set $[x]$. If we again fix $m$ to be the number of intermediate sets containing only $0$, there will be $\binom{c-m}{j-m} S(x,c-m)$ ways to create the remaining $c-m$ intermediate sets, namely, the number of ways of first partitioning the elements of $[x]$ into $c-m$ non-empty sets, and then choosing $j-m$ of those sets to contain the remaining $0$'s. Since $m$ can be at most $j$, and since $m$ partitions all possible sets of intermediate sets, we may sum over $m$ to get the final part of our formula, which completes the proof.
\end{proof}

We can now prove our result.

\begin{theorem}
The recursion in Theorem \ref{Section 3 recursion} together with the result in Lemma \ref{Section 3 segregated lemma} is enough to compute $p(A)$ for a general block $A$ of length $n$ corresponding to a (not necessarily admissible) set $P$. Furthermore, if $|P| = k$ then this can be done using no more than $O(k^2\log n+k^4)$ arithmetic operations.
\end{theorem}

\begin{proof}
First we will show that the recursion and above lemma are enough to compute $p(A)^i_j$, which will then imply the result for $p(A) = p(A)_0^0.$ Afterwards we will consider the run time. 

Suppose that $0^x$ is the maximal sequence of $0$'s at the start of $A$, and that $1^y$ is the maximal sequence of $1$'s directly following it. Then we may write $A$ as the block concatenation $BA'$ where $B = [0^x,1^y]$ is a segregated block and $A'$ is the rest of block $A$. Note that $A'$ must either be empty, or have fewer $1$'s in it than $A$ since $B$ will always take at least the first $1$ that appears in $A$.

If $A'$ is empty, then we have that $A$ started segregated and we are done by Lemma \ref{Section 3 segregated lemma}. 
If not, then we may apply Theorem \ref{Section 3 recursion} to write $p(A)_j^i$ in terms of $p(B)^{i'}_j$ and $p(A')^{i}_{j'}$, the former of which can be computed directly from Lemma \ref{Section 3 segregated lemma} and the later of which can be further broken down using Theorem \ref{Section 3 recursion} again. Since we know that $A'$ must either be empty or have fewer $1$'s than $A$, we will only have to repeat this recursion a finite number of times before $p(A)^i_j$ is written entirely in terms of numbers that Lemma \ref{Section 3 segregated lemma} can calculate, giving us our final answer.

To prove the bound for the run time, suppose that at each step the block break down was $A = A_0 = B_1A_1$, $A_1 = B_2A_2$, and so on until we have $A_{d-1} = B_{d}A_d$ where $A_d$ is the first $A_l$ to be empty. Since each $B_l$ contains at least one $1$, we have that $d$ is bounded by $k$, the number of pinnacles. 

It will be useful to find an upper bound on the different values for $i$ and $j$ that might appear in our calculations. Note that every time we apply the recursion to $p(A_l)_j^i$, we have that $i$ must be zero. This is because $i=0$ initially when computing $p(A)_0^0$, and we only ever apply the recursion again on the second factor, which has the same $i$ as at the start. However, the first factor will take the form $p(B_{l+1})_j^{i'}$ where $i'$ ranges over the values $0$ to $b$ where $b$ is the number of $1$'s in $A_{l+1}$. Therefore in the worst case, $i'$ may reach a maximum of $k-1$ when using Lemma \ref{Section 3 segregated lemma} to compute $p(B_{1})^{i'}_j$ and so in all cases will be bounded by $k$. Now suppose that we apply our recursion to $p(A_{l})^{i}_{j}$ and let $j'$ be the maximum value of $j$ that can arise on any factor of the form $A_{l+1}$ in the sum. We have that $j'$ is at most $\max\{j, b+1\}$ where $b+1$ is maximized on the first step when $b$ is the number of $1$'s in $A_{l+1}$, which is at most $k-1$. Since $j$ starts out at zero, and the value of $b$ falls each step, we have that the maximum value that $j$ can be is $k$.  

We now wish to find the total run time of using Lemma \ref{Section 3 segregated lemma} to calculate all the $p(B_{l})^i_j$. First, let $n_l$ be the number of $0$'s in block $B_l$ and note that each $n_l$ is clearly bounded by $n$, which is the original length of $A$. Then note that the quantity $c-m = y+i+1-m$ in Lemma \ref{Section 3 segregated lemma} can be at most $k+1$ since $y$ is the number of $1$'s in the block $B_l$ while $i$ can be at most the number of ones in the block $A_l$, which totals to no more than $k$. Therefore, we only have to calculate Stirling numbers of the form $S(n_l, k')$ where $l$ is bounded by $k$ and $k'$ is bounded by $k+1.$

Using the formula 
$$S(n_l,k') = \frac{1}{k'!}\sum_{d=0}^{k'} (-1)^{k'-d}\binom{k'}{d}{d}^{n_l}$$
we may first calculate all powers ${d}^{n_l}$ where $d \in [k+1]$ and $l \in [k]$ using fast exponentiation in at most $O(k^2\log n)$ arithmetic operations. We may then calculate all binomial expressions of the form $\binom{k'}{d}$ with $k' \leq k+1$ and $d \leq k'$ in $O(k^3)$ operations and store those too. From this information, we may calculate all Stirling numbers that will be needed in $O(k^3)$ additional time since the number of distinct $n_l, k'$ values are bounded by $k+1$. Furthermore, all other factorial and binomial expressions used in Lemma \ref{Section 3 segregated lemma}, along with the power of $2$, can be calculated and stored at this point in no more than $O(k^2\log n + k^3)$ time. Therefore, if we do this preparation first, we may calculate each $p(B_{l})^i_j$ in $O(k)$ time from the stored data. Since $i,j,$ and $l$ are each bounded by $k$, there will be at most $O(k^3)$ distinct such terms, we have that a look up table for all the $p(B_{l})^i_j$ can be generated in no more than $O(k^4)$ time. 

Given this look up table, we may then compute and store the $p(A_l)_j^0$ in $O(k)$ algebraic operations each by starting with $A_{m-1}$ and then accessing all the $p(B_{l})^i_j$ and previously computed $p(A_{l})^i_j$ to compute the next $A_{l-1}$. Since $l$ and $j$ are bounded by $k$, we have that this part of the algorithm takes at most $O(k^3)$ operations to get to $p(A_0)_0^0$, which is our final answer.

Therefore, the total run time is  $O(k^2\log n + k^3)$ to prepare for Lemma \ref{Section 3 segregated lemma}, plus $O(k^4)$ to create a look up table of the values in $p(B_k)^i_j$, plus a final $O(k^3)$ to use the recursion. Altogether, this results in a final run time of $O(k^2\log n+k^4)$ as advertised. 
\end{proof}

Before concluding, we present one more application of our recursion where we prove a version of a formula conjectured by Falque, Novelli, and Thibon in \cite{FNT_2021}. In their paper, the authors note that there is a map between admissible pinnacle sets in $[n]$ and forests of complete binary trees on $n$ vertices arranged as a sequence of trees from left to right, and such that the trees are labeled from left to right with nodes labeled in left suffix order. The map is given by the following algorithm. 

\begin{enumerate}
    \item Given some admissible pinnacle set $P = \{p_1 < p_2 < \ldots < p_k\} \subset [n]$, start with the forest consisting of a single unlabeled node. Let $i=n.$    
    \item Label the rightmost unlabeled node $i$. If $i\in P$, give this node two unlabeled children. Otherwise, create a new tree with one unlabeled node to the left of all existing trees. Let $i = i-1$.
    \item Repeat the previous step until a node has been given the label $1$, and then stop.
\end{enumerate}

To reverse this map, we start with any complete binary forest with trees arranged in a sequence from left to right and with nodes labeled as specified above. We then simply let $P$ be the set of labels of internal nodes. An example can be seen in Figure \ref{pinnacles to complete binary trees}.

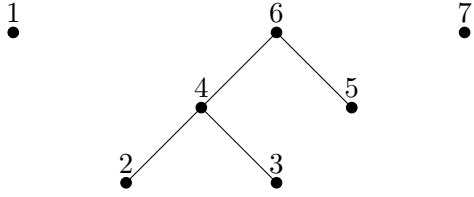
\begin{figure}
 \centering
\begin{tikzpicture}
\draw (3,3) -- (4,4) -- (5,3);
\draw (2,2) -- (3,3) -- (4,2);
        \filldraw[black] (6.5,4) circle (2pt) node[anchor=south]{7};
        \filldraw[black] (4,4) circle (2pt) node[anchor=south]{6};
        \filldraw[black] (5,3) circle (2pt) node[anchor=south]{5};
        \filldraw[black] (3,3) circle (2pt) node[anchor=south]{4};
        \filldraw[black] (4,2) circle (2pt) node[anchor=south]{3};
        \filldraw[black] (2,2) circle (2pt) node[anchor=south]{2};
        \filldraw[black] (0.5,4) circle (2pt) node[anchor=south]{1};
\end{tikzpicture}
    \caption{The complete binary forest corresponding to pinnacle set $P = \{4,6\}$ for $n=7$}
    \label{pinnacles to complete binary trees}
\end{figure}

It is not hard to show that these maps are well defined and that they are inverses of each other. In what follows, we will use the letter $F$ for forests, and $T$ for single trees. We may then write an admissible pinnacle set as a tree sequence $(T_1, T_2,  \ldots T_r)$. If a tree has only one vertex, denote it as $O$. We can also group consecutive trees together as an ordered forest; for instance, we may wish to say $(T_1, T_2, T_3, T_4) = (F_1, F_2)$ where $F_1 = (T_1)$ and $F_2 = (T_2, T_3, T_4)$. 

Using this notion of trees, we may use the notation $p(F)$ to denote the number of permutations in $S_n$ having pinnacle set represented by $(F)$. It was conjectured in \cite{FNT_2021} that for pinnacle set $F = (O,T_2,T_3, \ldots,T_r)$ we have the relation 
$$p(F) = p(O,T_2)P(O,O, T_3, \ldots, T_r).$$

Unfortunately, this is not true as can be seen by the counterexample when $n=4$ and the pinnacles set is $\{4\}$. However, a very similar result can be proved. 

\begin{theorem}
For a sequence of trees $F = (O,T_2,T_3, \ldots,T_r)$ encoding a pinnacle set, we have
$$p(F) = \frac{1}{2}p(O,T_2)p(O,O, T_3, \ldots, T_r).$$
\end{theorem}
\begin{proof}
Since the trees $F$ are labeled from left to right, we have that all nodes in $(O,T_2)$ will be smaller than all nodes in the rest of $F$. Therefore, we may consider a block composition $B_1B_2$ corresponding to this pinnacle set where $B_1$ corresponds only to the labels of the nodes in $(O,T_2)$ while $B_2$ has the rest. Then by our notation, we will have that $p(B_1) = p(O,T_2)$ and $p(B_2) = p(T_3, \ldots, T_r)$.

We then have by Theorem \ref{Section 3 recursion} that 
$$p(F) = p(B_1B_2) = \sum_{m=0}^b p(B_1)^mp(B_2)_{m+1}.$$
Now note that if $m>1$ we will have that $p(B_1)^m = 0$ since there will be too many pinnacles for the non-pinnacles to separate, and so this sum simplifies to
$$p(F) = p(B_1)p(B_2)_{1} + p(B_1)^1p(B_2)_2$$
First, note that $p(B_1) = p(B_1)^1$ since there is a simple bijection between the permutations they count. Namely, for any $\pi$ counted by $p(B_1)$ there will always exist exactly one index where two non-pinnacles are adjacent and where we may insert a new largest element to get a permutation counted by $p(B_1)^1$. Similarly, we my reverse this bijection by deleting the largest element from some $\pi'$ counted by $p(B_1)^1$. This will not create any new pinnacles because all non-pinnacles will still be adjacent to either the end of the permutation or to one of the other pinnacles. 

We now consider a block $B_3$ corresponding to $(O,O, T_3, \ldots, T_r)$, which will be $B_2$ with two zeros appended to the beginning. We claim that $p(B_3) = 2s$ where $s = p(B_2)_{1} + p(B_2)_2$. To see this, note that every permutation counted by $p(B_3)$ will either have the $1$ and $2$ adjacent, or separate. If they are adjacent, replace both by a single $0$ and subtract $2$ from all other elements. This will not destroy any pinnacles since neither $1$ nor $2$ were originally pinnacles, and so the result will be a permutation counted by $p(B_2)_{1}$. Since this map is clearly reversible up to the ordering of $1$ and $2$, we also see that this is a two-to-one mapping.

On the other hand, suppose that $1$ and $2$ are not adjacent. Then it is easy to see, by reasons similar to above, that if we replace both $1$ and $2$ by separate $0's$ and then subtract $2$ from all other elements we will end up with a permutation counted by $p(B_2)_{2}$. This will again be a two-to-one mapping and so we have that $p(B_3) = 2s$. 

Therefore, we have 
$$p(F) = p(B_1B_2) = p(B_1)p(B_2)_{1} + p(B_1)^1p(B_2)_2 = p(B_1)\left(p(B_2)_{1} + p(B_2)_2\right) = p(B_1)\frac{1}{2}p(B_3)$$
the last of which is equal to $\frac{1}{2}p(O,T_2)p(O,O, T_3, \ldots, T_r)$ as desired.
\end{proof}


\begin{thebibliography}{10}

\bibitem{AGRR_2020}
Ron~M. Adin, Ira~M. Gessel, Victor Reiner, and Yuval Roichman.
\newblock Cyclic quasi-symmetric functions.
\newblock {\em S\'{e}m. Lothar. Combin.}, 82B:Art. 67, 12, 2020.

\bibitem{Callan_2002}
David Callan.
\newblock Pattern avoidance in circular permutations, 2002.

\bibitem{CW_2019}
\'{E}va Czabarka and Zhiyu Wang.
\newblock Erd\H{o}s-{S}zekeres theorem for cyclic permutations.
\newblock {\em Involve}, 12(2):351--360, 2019.

\bibitem{DNPT_2018}
Robert Davis, Sarah~A. Nelson, T.~Kyle~Petersen, and Bridget~E. Tenner.
\newblock The pinnacle set of a permutation.
\newblock {\em Discrete Math.}, 341(11):3249--3270, 2018.

\bibitem{DHHIN_2021}
Alexander Diaz-Lopez, Pamela~E. Harris, Isabella Huang, Erik Insko, and Lars
  Nilsen.
\newblock A formula for enumerating permutations with a fixed pinnacle set.
\newblock {\em Discrete Math.}, 344(6):Paper No. 112375, 15, 2021.

\bibitem{DLMSSS_jun_2021}
Rachel Domagalski, Jinting Liang, Quinn Minnich, Bruce~E. Sagan, Jamie Schmidt,
  and Alexander Sietsema.
\newblock Cyclic pattern containment and avoidance, 2021.

\bibitem{DLMSSS_2021}
Rachel Domagalski, Jinting Liang, Quinn Minnich, Bruce~E. Sagan, Jamie Schmidt,
  and Alexander Sietsema.
\newblock Pinnacle set properties, 2021.

\bibitem{EB_2021}
Sergi Elizalde and Bruce Sagan.
\newblock Consecutive patterns in circular permutations, 2021.

\bibitem{FNT_2021}
Justine Falque, Jean-Christophe Novelli, and Jean-Yves Thibon.
\newblock Pinnacle sets revisited, 2021.

\bibitem{Fang_2021}
Wenjie Fang.
\newblock Efficient recurrence for the enumeration of permutations with fixed
  pinnacle set, 2021.

\bibitem{GLW_2018}
Daniel Gray, Charles Lanning, and Hua Wang.
\newblock Pattern containment in circular permutations.
\newblock {\em Integers}, 18B:Paper No. A4, 13, 2018.

\bibitem{GLW_2019}
Daniel Gray, Charles Lanning, and Hua Wang.
\newblock Patterns in colored circular permutations.
\newblock {\em Involve}, 12(1):157--169, 2019.

\bibitem{Poly_2021}
Alexios~P. Polychronakos.
\newblock Length and area generating functions for height-restricted motzkin
  meanders, 2021.

\bibitem{Rusu_2020}
Irena Rusu.
\newblock {Sorting permutations with a fixed pinnacle set}.
\newblock {\em The electronic journal of combinatorics}, 27(3), 2020.

\bibitem{RT_2021}
Irena Rusu and Bridget~Eileen Tenner.
\newblock Admissible pinnacle orderings.
\newblock {\em Graphs Combin.}, 37(4):1205--1214, 2021.

\bibitem{Strehl_1978}
Volker Strehl.
\newblock Enumeration of alternating permutations according to peak sets.
\newblock {\em J. Combinatorial Theory Ser. A}, 24(2):238--240, 1978.

\bibitem{Vella_2002}
Antoine Vella.
\newblock Pattern avoidance in permutations: linear and cyclic orders.
\newblock volume~9, pages Research paper 18, 43. 2002/03.
\newblock Permutation patterns (Otago, 2003).

\end{thebibliography}

\nocite{*}
\bibliographystyle{alpha}

\end{document}